\documentclass[oneside,12pt]{amsart}
\usepackage[T2A]{fontenc}
\usepackage[cp1251]{inputenc}
\usepackage{amssymb}
\usepackage{amsxtra}
\usepackage{amsmath}
\usepackage{amsthm}
\usepackage{amsthm, amsmath, amssymb, enumerate, paralist, calc, verbatim, stmaryrd, ifthen, comment, accents}
\usepackage{tikz,url}
\usepackage[T1]{fontenc}

\usepackage{hyperref}
\usepackage[all]{xy}

\usepackage[russian,english]{babel}

\newcommand{\GG}{G}   
\newcommand{\G}{\mathcal{G}}   
\newcommand{\GS}{\mathbf{S}}   
\newcommand{\GT}{T}   
\newcommand{\GP}{P}   

\newcommand{\A}{\mathcal{A}} 
\newcommand{\Ss}{\mathcal{S}}

\newcommand{\LL}{\mathcal{L}}

\newcommand{\CC}{\mathbb{C}}

\newcommand{\bt}{\bullet}

\newcommand{\barop}{\overline{\phantom{x}}}

\newcommand{\subhalf}{{\leavevmode \raise.5ex\hbox{\the\scriptfont0 1}\kern-.13em /\kern-.07em\lower.25ex\hbox{\the\scriptfont0 2}}}

\newcommand{\si}{\sigma}
\newcommand{\eps}{\epsilon}

\newcommand{\dash}{\nobreakdash-\hspace{0pt}}

\newcommand{\id}{\mathrm{id}}

\newlength{\dhatheight}

\DeclareMathOperator{\Aut}{Aut}

\DeclareMathOperator{\Char}{char}

\DeclareMathOperator{\End}{End}

\DeclareMathOperator{\cha}{char}

\DeclareMathOperator{\Instrl}{Inn}
\DeclareMathOperator{\ad}{ad}

\DeclareMathOperator{\Lie}{Lie}

\newcommand{\ZZ}{\mathbb Z}
\newcommand{\Gm}{{\mathbb G}_\mathrm{m}}
\DeclareMathOperator{\Cent}{Cent}

\DeclareMathOperator{\Der}{Der}

\newcommand{\lK}{\mathbf{\{}}
\newcommand{\rK}{\mathbf{\}}}
\newcommand{\VK}{V}
\newcommand{\KK}{K}

\newcounter{M}
\newcommand{\Mitem}[1][]{%
        \ifthenelse{\equal{#1}{}}
		{\item[(${\bf M}_{\refstepcounter{M}\theM}$)]}
		{\item[(${\bf M}_{\refstepcounter{M}\label{qitem:#1}\theM}$)]}
}

\allowdisplaybreaks

\newtheorem{theorem}{Theorem}[section]

\newtheorem{lemma}[theorem]{Lemma}

\newtheorem{corollary}[theorem]{Corollary}

{\theoremstyle{definition}
\newtheorem{remark}[theorem]{Remark}
\newtheorem*{remark*}{Remark}
\newtheorem{definition}[theorem]{Definition}
\newtheorem{notation}[theorem]{Notation}
\newtheorem{construction}[theorem]{Construction}

\newtheorem*{construction*}{Construction}
}

\begin{document}

\title[5-Graded simple Lie algebras, structurable algebras,
and Kantor pairs]{5-Graded simple Lie algebras,
structurable algebras, and Kantor pairs}

\author{Anastasia Stavrova}
\address{Chebyshev Laboratory, St. Petersburg State University,
14th Line V.O. 29B, 199178 Saint Petersburg, Russia}
\email{anastasia.stavrova@gmail.com}

\subjclass[2010]{16W10, 17B60, 17B45, 17B50, 17B70, 17C20}

\keywords{structurable algebra, Kantor pair,
5-graded Lie algebra, simple algebraic group, simple Lie algebra\medskip \\
The author is a winner of the contest ``Young Russian Mathematics''. The work was supported
by the RFBR grant 18-31-20044 and the Government of Russian Federation megagrant 14.W03.31.0030.}

\maketitle

\tableofcontents

\begin{abstract}
Relying on the classification of simple Lie algebras over algebraically closed fields of characteristic $>3$,
we show that any finite-dimensional central simple 5-graded Lie algebra over
a field $k$ of characteristic $\neq 2,3$ is a
simple Lie algebra of Chevalley type, i.e. a central quotient of the Lie algebra of a simple algebraic $k$-group.
As a consequence, we prove that all central simple structurable algebras and Kantor pairs over $k$ arise from 5-gradings
on simple Lie algebras of Chevalley type.
\end{abstract}

\maketitle

\section{Introduction}
A $\ZZ$-graded Lie algebra $\LL=\bigoplus_{i\in\ZZ}\LL_i$ is called 5-graded (respectively, 3-graded), if $\LL_i=0$
for all $i$ such that $|i|\ge 3$ (respectively, $|i|\ge 2$). In~\cite{A78,A79} Bruce Allison introduced
a method to construct 5-graded Lie algebras over a field $k$ of characteristic $\neq 2,3$
from a class of non-associative algebras with involution that he named
structurable algebras. This construction generalized the
Tits--Kantor--Koecher construction of 3-graded Lie algebras from Jordan algebras.
Allison showed that a structurable algebra $\A$ is central simple if and only if
the associated 5-graded Lie algebra $K(\A)$ is central
simple, and if, moreover, $\cha(k)=0$, then all central simple isotropic Lie algebras are of the form $K(\A)$. In particular,
this includes simple complex Lie algebras of types $G_2$, $F_4$, and $E_8$ which do not have non-trivial 3-gradings,
and thus cannot be obtained via the Tits--Kantor--Koecher construction.

If $k$ is an algebraically closed field of characteristic $0$, then all simple Lie algebras over $k$ are simple
Lie algebras of Chevalley type, that is, the simple
Lie algebras having a Chevalley basis with the same multiplication table as
the simple complex Lie algebras corresponding to the Dynkin diagrams of types $A_l-G_2$. If $\cha(k)>3$, these Lie algebras are also simple
except for the case $A_l$ with $\cha(k)|l+1$, in which case the central quotient of the corresponding Lie algebra is simple.
If $k$ is a not necessarily algebraically closed field of characteristic $\neq 2,3$, we will say that a finite-dimensional
central simple Lie algebra
$\LL$ over $k$ is of Chevalley type,
if it is isomorphic to one of the Chevalley type Lie algebras in the above sense after tensoring
with an algebraic closure of $k$. We will also call a Chevalley type Lie algebra isotropic, if it has a non-trivial
$\ZZ$-grading; this is equivalent to its automorphism group being an isotropic algebraic group in the sense of~\cite{Bo},
and agrees with Allison's definition of an isotropic Lie algebra. Over an algebraically closed field, every simple
Lie algebra of Chevalley type is isotropic.

In 2008 A. Premet and H. Strade~\cite{PrStr-VI} completed the
classification of finite-dimensional simple Lie algebras over an algebraically closed field of characteristic $>3$.
The resulting Block--Wilson--Premet--Strade classification theorem~\cite{Strade-bookI} establishes
that such a Lie algebra is either a simple Lie algebra of Chevalley type, or a Cartan type Lie algebra in characteristic $p>3$,
or a Melikian type Lie algebra in characteristic $p=5$. It is thus natural to ask which simple Lie algebras
in this list are of the form $K(\A)$, where $\A$ is a simple structurable algebra.
Analyzing possible $\ZZ$-gradings on Cartan and Melikian type Lie algebras,
we are able to establish that
simple structurable algebras are associated with simple Lie algebras of Chevalley type only.

\begin{theorem}\label{thm:1}
Let $k$ be a field of characteristic $\neq 2,3$, and let $\LL$ be a finite-dimensional
central simple Lie algebra over $k$. Then the following
are equivalent.
\begin{enumerate}
\item $\LL$ is an isotropic simple Lie algebra of Chevalley type.                                           
\item $\LL$ has a non-trivial $5$-grading.                                                       
\item There is a central simple structurable algebra $\A$ over $k$ such that $\LL\cong K(\A)$. 
\end{enumerate}
\end{theorem}

The implication $\mathit{(3)\implies (2)}$ of Theorem~\ref{thm:1} is obvious.
The implication $\mathit{(2)\implies (1)}$ is proved in Theorem~\ref{thm:Che-type} of \S~\ref{sec:5-gr}.
The implication $\mathit{(1)\implies (3)}$ is obtained in the very end of~\S~\ref{sec:5-class} as a corollary of
Theorem~\ref{thm:5-class}.

Apart from the characteristic $0$ case, Theorem~\ref{thm:1} has been known in the following cases.

If $\A$ is a central simple structurable algebra with trivial involution, then
it is a Jordan algebra, and $K(\A)$ is 3-graded. A theorem of A. Premet~\cite{Pre83,Strade-bookI} establishes
that any central simple $3$-graded Lie algebra over a field of characteristic $>3$ is a Chevalley type Lie algebra.
Alternatively, the fact that $K(\A)$ is
of Chevalley type follows from O. Loos' theory of algebraic groups defined by Jordan pairs~\cite{Lo-pairs,Lo-homog}.
This theory, moreover, extends to arbitrary characteristic, provided one uses
quadratic Jordan algebras instead of the usual ones.

Assuming $k$ is a field of characteristic $\neq 2,3,5$,  B. Allison and O. Smirnov~\cite{A78,Smi90}
classified all central simple structurable $k$-algebras
by direct methods, partially relying on the classification of simple Jordan algebras but
not on the classification of simple Lie algebras in positive characteristic. Given that the resulting list of central simple structurable algebras
is uniform in all characteristics, it is possible to conclude via case-by-case analysis that the	 central simple Lie algebras $K(\A)$
are exactly the central simple Lie algebras of Chevalley type, although this was not explicitly done in the above works.
In~\cite{St-thes,BdMS} we gave
a uniform, classification-free proof that any 5-graded central simple Lie algebra over an arbitrary field of characteristic
$\neq 2,3,5$ is of Chevalley type, extending the argument used in the above-mentioned theorem of A. Premet.
The same result also follows from~\cite[Theorem 3.1]{GGLN}, which shows that if $\cha(k)\neq 2,3,5$ then the Lie
algebras $K(\A)$ are non-degenerate, and the well-known fact (seemingly, due
to A. I. Kostrikin and A. Premet; see e.g.~\cite{Wil76} and~\cite[Corollary 12.4.7]{Strade-bookII}) that the only non-degenerate
Lie algebras in the Block--Wilson--Premet--Strade classification are the Chevalley type Lie algebras.

Apart from the above result, we have obtained in~\cite[Theorem 4.3.1]{BdMS} a
more precise version of Theorem~\ref{thm:1} for structurable division algebras. Namely, we established a
bijective correspondence between isotopy classes of central simple
structurable division algebras over any field $k$ of characteristic $\neq 2,3$, and classes of graded-isomorphic central
simple Chevalley type Lie algebras
of isotropic rank 1 over $k$. This correspondence does not extend to all central simple structurable algebras,
since Chevalley type Lie algebras of isotropic rank $>1$ may have several distinct 5-gradings corresponding to non-isomorphic
structurable algebras. This is evidenced by the second main result of the present paper, Theorem~\ref{thm:5-class},
that classifies all 5-gradings on central simple
Chevalley type Lie algebras $\LL$ over algebraically closed fields of characteristic $\neq 2,3$ that
give rise to an isomorphism $\LL\cong K(\A)$. This classification may serve as a first step towards
an alternative proof of the Allison--Smirnov classification theorem, eliminating the characteristic $\neq 5$ restriction.

Structurable algebras and 5-graded Lie algebras are also closely related to Kantor pairs~\cite{Kantor-cert,AF99,AllFauSmi,GGLN}.
Namely, every Kantor pair $(K_+,K_-)$ identifies with the pair of subspaces $(\G_{-1},\G_1)$
of a 5-graded Lie algebra $\G=\G(K_+,K_-)$,
called the standard graded embedding of $(K_+,K_-)$ in~\cite{AF99}, with a trilinear operation corresponding to the triple commutator
$[[x,y],z]$ in $\G$; see~\S~\ref{sec:ka-str} for the details. This makes our results
applicable to the study of simple Kanor pairs.
In particular, as a corollary of the implication $\mathit{(2)\implies (1)}$ of Theorem~\ref{thm:1}, we prove in
Corollary~\ref{cor:ka-degen} that
every simple Kantor pair over a field of characteristic $\neq 2,3$ is non-degenerate,
removing the restriction of invertibility of $5$ in~\cite[Theorem 3.1]{GGLN}.

All commutative rings we consider are assumed to be associative and unital. All algebras are finite-dimensional
over the respective scalars unless explicitly mentioned otherwise.

\section{Preliminaries on $\ZZ$-graded Lie algebras}


\begin{definition}
Let $R$ be a commutative ring, and let $\LL=\bigoplus\nolimits_{i\in\ZZ}\LL_i$ be a $\ZZ$-graded Lie algebra over $R$.
We say that $\LL$ is \emph{$(2n+1)$-graded}, if $\LL_i=0$ for all $i\in\ZZ$ such that $|i|>n$. The grading is~\emph{non-trivial},
if $\LL\neq\LL_0$.
\end{definition}

To simplify the notation, we will denote by $\Aut(\LL)$ the group of $R$-Lie algebra automorphisms
and by $\Der(\LL)$ the ring of $R$-Lie derivations of $\LL$, omitting the reference to $R$.
We also occasionally consider $\Aut(\LL)$ and $\Der(\LL)$ as functors on the category of $R$-algebras $R'$,
so that
\[
\Aut(\LL)(R')=\Aut_{R'}(\LL\otimes_R R'),\qquad \Der(\LL)(R')=\Der_{R'}(\LL\otimes_R R').
\]
Note that $\Aut(\LL)$ and $\Der(\LL)$ are naturally represented by closed $R$\dash subschemes of the affine $R$\dash scheme
of $R$-linear endomorphisms of $\LL$.

\begin{definition}\label{def:gradingtorus}
Let $\LL=\bigoplus\limits_{i\in\ZZ}\LL_i$ be a $\ZZ$-graded Lie algebra over a commutative ring $R$.
Consider  the 1-dimensional split $R$\dash subtorus $\GS\cong\Gm$ of $\Aut(\LL)$
defined as follows: for any $R$\dash algebra $R'$, any $t\in\Gm(R')$,
and any $i\in\ZZ$, $v\in\LL_i\otimes_{R} R'$, we set $t\cdot v=t^iv$. We call $\GS$
\emph{the grading torus of $\LL$}.
\end{definition}

\begin{definition}
$\LL=\bigoplus\limits_{i\in\ZZ}\LL_i$ be a $\ZZ$-graded Lie algebra over a commutative ring $R$.
The \emph{grading derivation}
on $\LL$ is the derivation $\zeta\in\Der(\LL)$ such that one has
\[
\zeta(x)=i\cdot x\quad\mbox{for any $i\in\ZZ$ and any $x\in\LL_i$}.
\]
If $\LL$ contains an element $\zeta$ such that $\ad(\zeta)$ is the grading derivation, we occasionally call $\zeta$
a grading derivation of $\LL$ by abuse of language.
\end{definition}

\begin{definition}\label{def:algebraic}
Let $\LL$ be a finite-dimensional $5$\dash graded Lie algebra over a commutative ring $R$ such that $2,3\in R^\times$.
We say that $x\in\LL$ is \emph{algebraic}, if the endomorphism
$\exp(x)=\sum\limits_{i=0}^4\frac 1{i!}\ad(x)^i$ is an automorphism
of $\LL$ as an $R$-Lie algebra. We say that $\LL$ is \emph{algebraic},
if all elements  $x\in\LL_i$, $i\neq 0$, are algebraic.
\end{definition}

\begin{remark}\label{rem:algebraic}
By~\cite[Lemma 3.1.7]{BdMS} any 3-graded Lie algebra, and hence any Jordan algebra over $R$ with $2,3\in R^\times $
is algebraic; any 5-graded Lie algebra, and hence any structurable algebra over $R$ with $2,3,5\in R^\times$ is algebraic.
(Despite that lemma is stated for algebras over a field, the proof is also valid over commutative rings.)
By~\cite[Theorem 4.2.8]{BdMS} any central simple structurable {\em division} algebra
over a field of characteristic $\neq 2,3$ is algebraic.
\end{remark}

\begin{definition}
Let $\LL$ be a Lie algebra. An element $x\in\LL$ is called an \emph{absolute zero divisor}, or a~\emph{sandwich},
if $[x,[x,\LL]]=0$. The Lie algebra $\LL$ is called~\emph{non-degenerate}, if it contains no non-zero absolute zero divisors.
\end{definition}

The following lemma elaborates on the proof of~\cite[Theorem 3.1]{GGLN}.
\begin{lemma}\label{lem:alg-nondegen}
Let $\LL$ be a simple 5-graded Lie algebra over a commutative ring $R$ such that $2,3\in R^\times$. If $\LL$ is algebraic,
it is non-degenerate.
\end{lemma}
\begin{proof}
Since $\LL$ is simple, $R$ is in fact a field. Let $U\subseteq\LL$ be a linear span of all absolute zero divizors.
Then $U$ is invariant under all Lie algebra automorphisms of $\LL$.
It is straightforward to see that $U$ is a Lie subalgebra. We claim that $U$ is a graded Lie subalgebra.
Indeed, by~\cite[Theorem 2.3]{GGLN} we have
$$
U=(U\cap(\LL_{-1}\oplus\LL_1))\oplus (U\cap(\LL_{-2}\oplus\LL_0\oplus\LL_2)).
$$
Since the characteristic of $R$ is $\neq 2,3$,
it contains at least 4 distinct invertible elements. Considering the action of the grading torus on
$U\cap(\LL_{-1}\oplus\LL_1)$ and $U\cap(\LL_{-2}\oplus\LL_0\oplus\LL_2)$, one readily sees that they contain
homogeneous components of all their elements.

Since $\LL$ is algebraic, $\exp(x)=\sum\limits_{i=0}^4\frac 1{i!}\ad(x)^i$ is a Lie automorphism of $\LL$ for all
$x\in\LL_i$,
$i\neq 0$. Considering the homogeneous components of of $\exp(x)(u)$ for an $u\in U_i$, $-2\le i\le 2$,
we see that $[x,u]\in U$. Since $\LL$ is simple, it is generated by $\LL_i$, $i\neq 0$
(see e.g.~\cite[Lemma 4.1.5]{BdMS}). Hence $U$ is an ideal of $\LL$,
and hence $U=0$ or $U=\LL$.
By~\cite[Theorem 1]{Zel-abs} $U$ is nilpotent, hence $U=0$.
\end{proof}

The following results relate algebraicity to the classification of simple Lie algebras.

\begin{theorem}\cite[Theorem 4.1.8]{BdMS}\label{thm:BdMS-premet}
Let $\LL$ be an algebraic central simple $5$\dash graded Lie algebra over
a field $k$ of characteristic different from $2,3$, such that $\LL\neq\LL_0$.
Then the algebraic $k$\dash group $\GG=\Aut(\LL)^\circ$
is an adjoint absolutely simple group of $k$\dash rank $\geq 1$, satisfying $\LL=[\Lie(\GG),\Lie(\GG)]$.
\end{theorem}

\begin{lemma}\cite[Lemma 4.2.4 (i)]{BdMS}\label{lem:algebraic}
Let $k$ be a field, $\Char k\neq 2,3$. Let $\GG$ be an adjoint simple algebraic group over $k$.
Let $\LL=\Lie(\GG)$ be its Lie algebra.
Let $\LL=\bigoplus\limits_{i=-2}^2\LL_i$ be any $5$\dash grading on $\LL$
such that $\LL_1\oplus\LL_{-1}\neq 0$.
The $5$-graded Lie algebra $\LL$ is algebraic.
\end{lemma}

\begin{theorem}\label{thm:Che-alg}
Let $\LL$ be a central simple $5$\dash graded Lie algebra over
a field $k$ of characteristic different from $2,3$, such that $\LL\neq\LL_0$. Then $\LL$ is of Chevalley type
if and only if $\LL$ is algebraic.
\end{theorem}
\begin{proof}
By~\cite[Lemma 3.1.6]{BdMS} a 5-graded Lie algebra $\LL$ over $k$ is algebraic, if and only if $\LL\otimes_k\bar k$
is algebraic, where $\bar k$ is the algebraic closure of $k$. By definition,
a Lie algebra $\LL$ is of Chevalley type if and only if $\LL\otimes_k\bar k$ is of Chevalley type. Thus, we can assume from the
start that $k$ is algebraically closed. Over an algebraically closed field $k$, simple Lie algebras of Chevalley type
coincide with $[\Lie(G),\Lie(G)]$, where $G$ is an adjoint simple algebraic group over $k$ (cf.~\cite[Ch. III]{Sel67} and~\cite[Lemma 4.1.6]{BdMS}),
and all these Lie algebras are algebraic by Lemma~\ref{lem:algebraic}. The converse holds by Theorem~\ref{thm:BdMS-premet}.
\end{proof}

\section{5-Graded simple Lie algebras}\label{sec:5-gr}

In this section we establish the implication $\mathit{(2)\implies (1)}$ of Theorem~\ref{thm:1}. Namely, we prove the following
theorem.

\begin{theorem}\label{thm:Che-type}
Let $\LL$ be a central simple $5$\dash graded Lie algebra over
a field $k$ of characteristic different from $2,3$, such that $\LL\neq\LL_0$.
Then $\LL$ is of Chevalley type.
\end{theorem}

The proof relies on the Block--Wilson--Premet--Strade classification of simple Lie algebras over
algebraically closed fields of characteristic $\neq 2,3$,
and in what follows we use the notation of~\cite{Strade-bookI} for the classes of central simple Lie algebras
that arise in that classification.

\begin{lemma}\label{lem:L-1}
Let $\LL\neq\LL_0$ be a simple $5$\dash graded Lie algebra over
a field $k$ of characteristic $\neq 2,3$, and let $\mathbf{S}$ be the corresponding grading torus. Let
$\LL=\bigoplus_{i=-r}^s\LL_{[i]}$ be another $\ZZ$-grading on $\LL$ such that $\bigoplus_{i=-r}^{-1}\LL_{[i]}$ is
generated as a Lie algebra by $\LL_{[-1]}\neq 0$. If
the second grading is preserved by $\mathbf{S}(k)$, then $\LL_{[-1]}\not\subseteq\LL_0$.
\end{lemma}
\begin{proof}
Since $\mathbf{S}$ preserves the second grading, we have
$$
\LL_{[i]}=\bigoplus_{j=-2}^2 (\LL_{[i]}\cap\LL_j)
$$
for all $-r\le i\le s$.
Assume  that $\LL_{[-1]}\subseteq\LL_{0}$. Then
$\LL_{[-i]}\subseteq\LL_0$ for all $1\le i\le r$.
For any $x\in\LL_j$, $-2\le j\le 2$, we have $t\cdot x=t^j x$ for any $t\in\mathbf{S}(k)$.
Since $\Char(k)\neq 2,3$, there is $t\in k$ such that $t^{\pm 1}\neq 1$ and $t^{\pm 2}\neq 1$.
Then $\LL_j\subseteq\bigoplus_{i\ge 0}\LL_{[i]}$ for all $j\neq 0$.
Since $\LL$ is simple, it is generated by $\LL_{j}$, $j\neq 0$, hence $\LL=\bigoplus_{i\ge 0}\LL_{[i]}$.
However, this contradicts $\LL_{[-1]}\neq 0$.
\end{proof}

\begin{lemma}\label{lem:graded-cartan}
Let $\LL$ be one of the simple graded Cartan type Lie algebras $X(m,\underline{n})^{(2)}$, $X\in\{W,S,H,K\}$,
in the notation of~\cite{Strade-bookI}, over an algebraically closed field
$k$ of characteristic $5$. Then $\LL$ does not have a non-trivial $5$-grading.
\end{lemma}
\begin{proof}
Let $\LL=\bigoplus_{i=-2}^2\LL_i$ be a non-trivial $5$-grading on $\LL$, and let $\LL=\bigoplus_{i=-r}^s\LL_{[i]}$
be the standard grading on $\LL$. Note that $r=1$ for $X=W,S,H$, and $r=2$, $\dim(\LL_{[-2]})=1$ for $X=K$.

By~\cite[Theorem 7.4.1]{Strade-bookI} we can assume that the grading torus $\mathbf{S}$ of the $5$-grading	 preserves the
standard grading.
Then by Lemma~\ref{lem:L-1} the $5$-grading on $\LL_{[-1]}$ is non-trivial.
Since $\LL_{[-1]}$
is contained in the restricted subalgebra $X(m,\underline{1})^{(2)}$ of $\LL$,
the induced $5$-grading on the restricted subalgebra is non-trivial, and hence we can assume
that $\LL=X(m,\underline{1})^{(2)}$ from the start.

We have $\LL_{[-1]}=\ad(\LL_{[-1]})^{s+1}(\LL_{[s]})$ (see e.g.~\cite[Ch. 4, proofs of Theorems 2.4, 3.5, 4.5, 5.5]{SF}).
We claim that $\LL_{[-1]}$ contains an element $y\neq 0$ of non-zero 5-grading such that
$$
\ad(y)^4\ad(\LL_{[-1]})^{s-3}(\LL_{[s]})=0.
$$
Since any $y\in\LL_{-2}\oplus\LL_2$ satisfies $\ad(y)^4=0$, we only need to consider the case where
$\LL_{[-1]}\subseteq\LL_{-1}\oplus\LL_0\oplus\LL_1$. Since the 5-grading on $\LL_{[-1]}$ is non-trivial,
there is a non-zero element $y\in\LL_{[-1]}\cap(\LL_{-1}\oplus\LL_{1})$.
Then $\ad(y)^4(\LL)\subseteq\LL_{\pm 2}$, hence
$\ad(y)^4\ad(\LL_{[-1]})^{s-3}(\LL_{[s]})\subseteq \LL_{[-1]}\cap \LL_{\pm 2}=0$, as required.

Assume without loss of generality that $y\in\LL_1$ or $y\in\LL_2$.
By~\cite[Theorem 2]{Wilson-auto} the space $\LL_{[-1]}$
is an irreducible representation for the group of Lie algebra automorphisms of $\LL$ preserving the standard grading.
Hence $\LL_{[-1]}$ has a basis consisting of elements $x_1=y$, $x_2,\ldots,x_{\dim(\LL_{[-1]})}$ such that
$\ad(x_i)^4\ad(\LL_{[-1]})^{s-3}(\LL_{[s]})=0$ for all i.
Any $z\in\LL_{[-1]}$ can be written as $z=\sum_{i=1}^{\dim(\LL_{[-1]})}\lambda _ix_i$. Any
endomorphism $\prod_{k=1}^{s+1}\ad(z_k)\in\ad(\LL_{[-1]})^{s+1}$, $z_k\in\LL_{[-1]}$,
is a linear combination of products of basic derivations $\ad(x_i)$ of length $s+1$.
Assume first that $X=W,S,H$, so that $r=1$, and $\LL_{[-1]}$ is abelian.
Then we can reorder the terms of such a product of basic derivations so that all occurences of the same basic
derivation are brought together. Since
$\ad(x_i)^4\ad(\LL_{[-1]})^{s-3}(\LL_{[s]})=0$, we conclude that
every non-zero such product contains $\le 3$ occurences of each $\ad(x_i)$. In
all three cases $X=W,S,H$ we have $\dim(\LL_{[-1]})=m$.
If $X=W$, then $m\ge 1$ and $s=4m-1$. If $X=S$, then $m\ge 2$ and $s=4m-2$.
If $X=H$, then $s=4m-3$ and $m\ge 2$. Therefore, $s+1\ge 3m$, hence the only possible
non-zero product of length $s+1$ is the one that contains exactly 3 occurences of each $\ad(x_i)$.
Consequently, $\LL_{[-1]}\subseteq\ad(y)^3(\LL)\subseteq \LL_1\oplus\LL_2$. But then the total 5-grading
of a product of $s+1$ derivations $\ad(x)$, $x\in\LL_{[-1]}$, is at least $s+1$. In all cases $s+1\ge 4$.
If $s+1>4$, then this implies
$\ad(\LL_{[-1]})^{s+1}=0$, a contradiction. If $s+1=4$, then $\ad(\LL_{[-1]})^{s+1}(\LL_{[s]})\subseteq\LL_2$,
so that all $x\in\LL_{[-1]}$ have degree $2$, and then again $\ad(\LL_{[-1]})^{s+1}(\LL_{[s]})=0$, a contradiction.

It remains to consider the case $X=K$, where $\LL_{[-1]}$ is not
abelian, $r=2$, $\dim(\LL_{[-2]})=1$. In this case we have
$m\ge 3$, $\dim(\LL_{[-1]})=m-1$, and $s=4m$, if $m+3\equiv 0\pmod{5}$, and $s=4m+1$ otherwise.
As before, consider a product of basic derivations $\ad(x_i)$ of length $s+1$.
We can still reorder all basic derivations
 in any way we want, however, in the process
we add new products where
some pairs $\ad(x_j)\cdot\ad(x_i)$ will be replaced by $\ad([x_i,x_j])\in\ad(\LL_{[-2]})$. Since $\LL_{[-2]}$
is 1-dimensional, we may obtain $\le 4$ such new terms $\ad([x_i,x_j])$,
and after that the new products will always be 0. These additional products thus contain $\ge s+1-2\cdot 4=s-7$
basic derivations $\ad(x_i)$. If $m+3\equiv 0\pmod{5}$, then $m\ge 7$, and $s-7=4m-7=3(m-1)+m-4>3(m-1)$, so
such a product of $\ge s-7$ derivaions is necessarily zero. If $m+3\not\equiv 0\pmod{5}$, then $s-7=4m-6=3(m-1)+m-3$
may still be equal $3(m-1)$, if $m=3$. Then we conclude as in the previous case that $\LL_{[-1]}\subseteq
\LL_1\oplus\LL_2$ (and automatically $\LL_{[-2]}\subseteq\LL_2$). Hence the total degree of a product
of basic derivations is $\ge s-7\ge 3(m-1)=6$, hence
$\ad(\LL_{[-1]})^{s+1}(\LL_{[s]})=0$, a contradiction.
\end{proof}

\begin{theorem}\label{thm:main}
Let $\LL$ be a central simple $5$\dash graded Lie algebra over
an algebraically closed field $k$ of characteristic different from $2,3$, such that $\LL\neq\LL_0$.
Then $\LL$ is a simple Lie algebra of Chevalley type.
\end{theorem}
\begin{proof}
If $\Char k\neq 5$, then $\LL$ is a Chevalley Lie algebra by Remark~\ref{rem:algebraic} combined with Theorem~\ref{thm:Che-alg}.
Assume $\Char k=5$.
According to the Block--Wilson--Premet--Strade classification theorem~\cite{Strade-bookI},
it is enough to check that $\LL$ is not of Cartan or Melikian type.

Assume first that $\LL$ is a simple Lie algebra of Cartan type~\cite[Definition 4.2.4]{Strade-bookI}.
Let $\LL=\LL_{(-r)}\supseteq\ldots\supseteq\LL_{(s)}$
be a standard filtration of $\LL$.
By~\cite[Theorem 4.2.7 (3)]{Strade-bookI} the standard filtration is invariant under
all automorphisms of $\LL$. In particular, it is invariant under
the grading torus $\mathbf{S}$ of the $5$-grading.
Let
$$
Gr(\LL)=\bigoplus_{i=-r}^sGr(\LL)_{[i]}
$$
be the associated graded Lie algebra.  Hence $Gr(\LL)$ carries an induced $5$-grading. The induced $5$-grading is
non-trivial, since there is $0\neq x\in\LL_i$, $i\neq 0$, and $t\in\mathbb{F}_5^\times\subseteq\mathbf{S}(k)$ such that
$t\cdot x=t^ix\neq x$, and hence $t$ acts non-trivially on the image of $x$ in $Gr(\LL)$.

The derived series of $Gr(\LL)$ also inherits the $5$-grading,
hence $Gr(\LL)^{(\infty)}$ is $5$-graded.
We show that the induced $5$-grading on $Gr(\LL)^{(\infty)}$ is also non-trivial.
Indeed, if it were trivial, then $Gr(\LL)^{(\infty)}\subseteq Gr(\LL)_0$.
Since $Gr(\LL)$ is also a Lie algebra of Cartan type in the sense
of~\cite[Definition 4.2.4]{Strade-bookI}, $Gr(\LL)^{(\infty)}$ is simple by~\cite[Theorem 4.2.7 (1)]{Strade-bookI}. Then
we have
$$
Gr(\LL)^{(\infty)}\subseteq (Gr(\LL)_0)^{(\infty)}\subseteq Gr(\LL)^{(\infty)},
$$
which implies $Gr(\LL)^{(\infty)}=(Gr(\LL)_0)^{(\infty)}$.
Then by~\cite[Lemma 4.2.5]{Strade-bookI} we have $Gr(\LL)=Gr(\LL)_0$, which contradicts the non-triviality
of the $5$-grading on $Gr(\LL)$.

By~\cite[Theorem 4.2.7 (2)]{Strade-bookI}
we have $Gr(\LL)^{(\infty)}=X(m,\underline{n})^{(2)}$, $X=W,S,H,K$.
By Lemma~\ref{lem:graded-cartan} these algebras do not have non-trivial $5$-gradings.

Now let $\LL=M(2,n_1,n_2)$ be a simple Lie algebra of Melikian type. By~\cite[Theorem 1.2]{BKMcG}
we can assume that the grading torus corresponding to the $5$-grading under consideration preserves the standard
granding $\LL=\bigoplus_{i=-3}^s\LL_{[i]}$ of $\LL$. Then the grading derivation $\zeta$ corresponding to the $5$-grading
is a homogenous derivation of $\LL$. Let $W(2,n_1,n_2)$ be the standard simple Witt subalgebra of $\LL$. Then
by~\cite[Theorem 7.1.4]{Strade-bookI} any homogeneous derivation of $\LL$ acts non-trivially on $W(2,n_1,n_2)$, and hence
this subalgebra carries a non-trivial $5$-grading induced by $\zeta$. However, this is not possible by Lemma~\ref{lem:graded-cartan}.
\end{proof}

\begin{proof}[Proof of Theorem~\ref{thm:Che-type}]
Let $\bar k$ be the algebraic closure of $k$. Then $\LL\otimes_k\bar k$ is a central simple Lie algebra over $\bar k$
with a non-trivial $5$-grading. Then by Theorem~\ref{thm:main}
$\LL\otimes_k\bar k$ is a Lie algebra of Chevalley type. Then it is algebraic by Theorem~\ref{thm:Che-alg}.
Hence $\LL$ is algebraic, since $\LL$ embeds into $\LL\otimes_k\bar k$. Again by Theorem~\ref{thm:Che-alg}
we conclude that $\LL$ is of Chevalley type.
\end{proof}

\section{Simple structurable algebras and Kantor pairs}\label{sec:ka-str}

\begin{definition}
        A {\em structurable algebra} over a field $k$ of characteristic not~$2$ or $3$ is a finite-dimensional, unital $k$\dash algebra with involution
        $(\A,\bar{\ })$ such that
        \begin{equation}\label{struct id}
                [V_{x,y}, V_{z,w}] = V_{\{x,y,z\},w} - V_{z,\{y,x,w\}}
        \end{equation}
        for $x,y,z,w \in \A$, where the left hand side denotes the Lie bracket of the two operators, and where
        \[ V_{x,y}z := \{x \ y \ z\} := (x\overline{y})z + (z\overline{y})x - (z\overline{x})y . \]

        For all $x,y,z \in \A$, we write $U_{x,y}z:=V_{x,z}y$ and $ U_xy:=U_{x,x}y$.
        The trilinear map $(x,y,z) \mapsto \{ x \ y \ z \}$ is called the {\em triple product} of the structurable algebra.
\end{definition}

In \cite{A78} and \cite{A79}, a structurable algebra is defined as an algebra with involution such that
\begin{align}\label{eqdef}
        [T_z,V_{x,y}]=&V_{T_z x,y}-V_{x,T_{\overline{z}}y}
\end{align}
for all $x,y,z \in \A$ with $T_x:=V_{x,1}$.
The equivalence of \eqref{struct id} and \eqref{eqdef} follows from \cite[Corollary 5.(v)]{A79}.

\begin{definition}
        Let $(\A,\bar{\ })$ be a structurable algebra; then $\A=\mathcal{H}\oplus\Ss$ for
        \[ \mathcal{H}=\{h\in \mathcal{A}\mid \overline{h}=h\} \quad \text{and} \quad \Ss=\{s\in \mathcal{A}\mid \overline{s}=-s\}. \]
        The elements of $\mathcal{H}$ are called {\em hermitian elements},
        the elements of $\Ss$ are called {\em skew-hermitian elements} or briefly {\em skew elements}.
\end{definition}

As usual, the commutator and the associator are defined as
\[ [x,y]=xy-yx,\quad [x,y,z]=(xy)z-x(yz),\]
for all $x,y,z\in \A$.
For each $s \in \Ss$, we define the operator $L_s \colon \A\to\A$ by
\[ L_s x := sx . \]
The following map is of crucial importance in the study of structurable algebras:
\[\psi \colon \A\times \A\to\Ss \colon (x,y)\mapsto x\overline{y}-y\overline{x}.\]

\begin{definition}
        An {\em ideal} of $\A$ is a two-sided ideal stabilized by $\barop$.
        A structurable algebra $(\A,\barop)$ is {\em simple} if its only ideals are $\{0\}$ and $\A$,
        and it is called {\em central} if its center
        \begin{multline*}
            Z(\A,\barop)=Z(\A)\cap \mathcal{H}
                =\{c\in\A\mid [c,\A]=[c,\A,\A]=[\A,c,\A]=[\A,\A,c]=0\}\cap \mathcal{H}
        \end{multline*}
        is equal to $k1$.
\end{definition}

Recall the generalization of the
Tits--Kantor--Koecher construction that associates to any structurable algebra $\A$  a $5$-graded Lie algebra
$K(\A)$.

Let $\End(\A)$ be the ring of $k$\dash linear maps from $\A$ to $\A$.
For each $A\in \End(\A)$, we define new $k$\dash linear maps
\begin{align*}
A^\eps&=A-L_{A(1)+\overline{A(1)}},\\
A^\delta&=A+R_{\overline{A(1)}},
\end{align*}
where $L_x$ and $R_x$ denote left and right multiplication by an element $x \in \A$, respectively.

\begin{construction}\label{def:Lie alg}\cite[\S 3]{A79}
Let $(\A,\bar{\ })$ be a structurable algebra over a field $k$ of characteristic not~$2$ or $3$.
Consider two copies $\A_+$ and $\A_-$ of $\A$ with corresponding isomorphisms $\A \to \A_+ \colon x \mapsto x_+$
and $\A \to \A_- \colon x \mapsto x_-$, and let $\Ss_+\subseteq A_+$ and $\Ss_-\subseteq A_-$ be the corresponding subspaces of skew elements.
Let $\Instrl(\A)$ be the $k$-subspace of $\End(\A)$ spanned by all
$V_{x,y}$, $x,y\in\A$.
The 5-graded Lie algebra $K(\A)$ associated to $\A$ is the vector space
\[ K(\A)=\Ss_-\oplus \A_-\oplus \Instrl(\A) \oplus \A_+ \oplus \Ss_+ \]
with the 5-graded components
\begin{multline*}
	K(\A)_{-2}=\Ss_-,\quad K(\A)_{-1}=\A_-,\quad K(\A)_{0}=\Instrl(\A),\quad
		K(\A)_{1}=\A_+,\quad K(\A)_{2}=\Ss_+,
\end{multline*}
and the following Lie bracket.
\begin{alignat*}{2}
\intertext{$\bt\ [\Instrl,K(\A)]$}
[V_{a,b},V_{a',b'}]&:=V_{a,b}V_{a',b'}-V_{a',b'}V_{a,b}=V_{\{a,b,a'\},b'}-V_{a',\{b,a,b'\}} &\in \Instrl(\A)&\\
[V_{a,b},x_+] &:= (V_{a,b}x)_+ \in \A_+  & [V_{a,b},y_-] &:= (V_{a,b}^\eps y)_- =(-V_{b,a}y)_-\in \A_-\\
[V_{a,b},s_+] &:= (V_{a,b}^\delta s)_+\ = -\psi(a,sb)_+\in \Ss_+ & [V_{a,b},t_-] &:= (V_{a,b}^{\eps\delta} t)_- \ =\psi(b,ta)_-\in \Ss_-\\
\intertext{$\bt\ [\Ss_\pm,\A_\pm]$}
[s_+,x_+] &:= 0& [t_-,y_-] &:= 0\\
[s_+,y_-] &:= (sy)_+\in \A_+& [t_-,x_+] &:= (tx)_-\in \A_-\\
\intertext{$\bt\  [\A_\pm,\A_\pm]$}
[x_+,y_-] &:= V_{x,y}\in \Instrl(\A)&[y_-,y_-'] &:= \psi(y,y')_-\in \Ss_-\\
[x_+,x_+'] &:= \psi(x,x')_+\in \Ss_+\\
\intertext{$\bt \ [\Ss_\pm,\Ss_\pm]$}
[s_+,s_+']&:=0,&[t_-,t_-']&:=0\\
[s_+,t_-]&:=L_{s}L_{t}\in \Instrl(\A)
\end{alignat*}
for all $x,x',y,y'\in \A$, all $s,s',t,t'\in \Ss$, and all $V_{a,b},V_{a',b'}\in \Instrl(\A)$.
\end{construction}

In the case where $\A$ is a Jordan algebra, we have $\Ss=0$, and thus the Lie algebra $K(\A)$ has a $3$-grading;
in this case $K(\A)$ is exactly the Tits--Kantor--Koecher construction of a Lie algebra from a Jordan algebra.

It is shown in \cite[\S 5]{A79} that the structurable algebra $\A$ is simple if and only if $K(\A)$ is a simple Lie algebra,
and that $\A$ is central if and only if $K(\A)$ is central.

\begin{definition}
A structurable algebra $\A$ is called \emph{algebraic}, if $K(\A)$ is algebraic.
\end{definition}

\begin{definition}
            Let $\A$ be a structurable algebra over a field $k$ of characteristic $\neq 2,3$.
            An element $x\in\A$ is called an \emph{absolute zero divisor} if $U_xy=0$ for any $y\in\A$.
            The algebra $\A$ is called \emph{non-degenerate} if it has no non-trivial absolute zero divisors.
\end{definition}
If an element $x\in K(\A)_\si=\A_\si$ is an absolute zero divisor of $K(\A)$,
then it is represented by an absolute zero divisor of $\A$;
this follows from the fact that by Definition~\ref{def:Lie alg},
\[ [x_\sigma, [x_\sigma, y_{-\sigma}]] = -V_{x,y} x \in \A_\sigma \]
for all $x,y \in \A$.

The following theorem strengthens~\cite[Theorem 4.1.1]{BdMS}.

\begin{theorem}\label{thm:structurable}
Let $\A$  be a central simple structurable algebra over a field $k$ of characteristic different from $2,3$.
Then $\A$ is algebraic and non-degenerate.
The algebraic $k$\dash group $\GG=\Aut(K(\A))^\circ$
is an adjoint absolutely simple group of $k$\dash rank $\geq 1$, and $K(\A)=[\Lie(\GG),\Lie(\GG)]$.
\end{theorem}
\begin{proof}
The Lie algebra $K(\A)$ is a central simple Lie algebra with a non-trivial 5-grading, hence it is
a Lie algebra of Chevalley type by Theorem~\ref{thm:Che-type}. Hence it is algebraic by Theorem~\ref{thm:Che-alg}.
By Lemma~\ref{lem:alg-nondegen} $K(\A)$ is non-degenerate, and then $\A$
is non-degenerate.
The remaining claims were established in~\cite[Theorem 4.1.1]{BdMS} under the additional assumption that $\A$ is algebraic.
\end{proof}

5-Graded Lie algebras and structurable algebras are closely related to Kantor pairs.

\begin{definition}[\cite{AF99}]\label{def:Kantor}
A \emph{Kantor pair} is a pair of finite-dimensional vector spaces $(K_+,K_-)$ over $k$
equipped with a trilinear product
\[
\lK\cdot,\cdot,\cdot\rK \colon K_\si\times K_{-\si}\times K_\si\to K_\si,\quad \si\in\{-1,1\},
\]
satisfying the following two identities:
        \begin{compactitem}
           \item[(KP1)]  $[\VK_{x,y}, \VK_{z,w}] = \VK_{\lK x,y,z\rK,w} - \VK_{z,\lK y,x,w \rK}$;
           \item[(KP2)]  $\KK_{a,b}\VK_{x,y}+\VK_{y,x}\KK_{a,b}=\KK_{\KK_{a,b}x,y}$;
        \end{compactitem}
where $\VK_{x,y}z:=\lK x,y,z\rK$ and $\KK_{a,b}z:=\lK a,z,b \rK -\lK b,z,a \rK$.
\end{definition}

If $\LL=\bigoplus_{i=-2}^2\LL_i$ is a 5-graded Lie algebra and $\zeta$ is its grading derivation, then the pair $(\LL_1,\LL_{-1})$ is a Kantor pair with
respect to the triple product $\{x,y,z\}=-[[x,y],z]$ for all $x,z\in\LL_\si$, $y\in\LL_{-\si}$
(alternatively, one may set $\{x,y,z\}=[[x,y],z]$, see e.g.~\cite{GGLN}).
Conversely, by~\cite[Theorem 7 and p. 535]{AF99},
for any Kantor pair $(K_+,K_-)$ there exists a unique up to isomorphism $5$-graded Lie algebra $\G=\G(K_+,K_-)$,
called the standard embedding of $(K_+,K_-)$,
such that the associated Kantor pair $(\G_+,\G_-)$ is isomorphic to $(K_+,K_-)$, $\G_{2\si}=[\G_\si,\G_\si]$,
$\G_0=k\zeta+[\G_\si,\G_{-\si}]$, and $\ad:\G_{-2}\oplus\G_0\oplus\G_2\to \End(\G_1\oplus\G_{-1})$ is injective.

If $\A$ is a structurable algebra over $k$, then the pair $(\A,\A)$ with the triple product $\lK x,y,z\rK=2\{x,y,z\}_{\A}$
is a Kantor pair (we double the triple product in accordance with the conventions of~\cite{AF99}). It is easy to see
that the corresponding standard embedding $\G(\A,\A)$ is graded-isomorphic to $K(\A)+k\zeta$~\cite[\S~3.2]{BdMS}.

\begin{definition}
An \emph{ideal} of a Kantor pair is a pair of subspaces $I_\si\subseteq K_\si$ such that
$$
\{I_\si,V_{-\si},V_\si\}+\{V_\si,I_{-\si},V_\si\}+\{V_\si,V_{-\si},I_\si\}\subseteq I_\si
$$
for $\si=\pm 1$. A Kantor pair with non-zero product is called \emph{simple}, if it has no non-trivial ideals.
The \emph{centroid} of a Kantor pair is the set of all $(c_+,c_-)\in\End(K_+)\oplus\End(K_-)$ such that
$$
c_\si(\{x_\si,y_{-\si},z_\si\})=\{c_\si(x_\si),y_{-\si},z_\si\}=\{x_\si,c_{-\si}(y_{-\si}),z_\si\}=
\{x_\si,y_{-\si},c_\si(z_\si)\}
$$
for all $x_\si,z_\si\in K_\si$, $y_{-\si}\in K_{-\si}$. A Kantor pair is called \emph{central}, if the centroid
coincides with the set of scalar pairs $(c,c)$, $c\in k$.
\end{definition}

According to the following result, every $5$-grading on a simple Lie algebra of Chevalley type over a field
of characteristic $\neq 2,3$ arises from a Kantor pair. In particular, every isotropic adjoint simple
algberaic group can be constructed from a simple Kantor pair.

\begin{lemma}\label{lem:4.3.3}\cite[Lemma 4.3.3]{BdMS}
Let $k$ be a field of characteristic different from $2,3$. Let $\GG$ be an adjoint simple algebraic group over
$k$. Let $\LL=\Lie(\GG)$ be its Lie algebra, and let $\LL=\bigoplus\limits_{i=-2}^2\LL_i$ be any $5$\dash grading on
$\LL$ such that $\LL_1\oplus\LL_{-1}\neq 0$.
Then $(\LL_1,\LL_{-1})$ is a central simple Kantor pair with respect
to the triple product operation $\LL_\sigma\times \LL_{-\sigma}\times \LL_{\sigma}\to \LL_\sigma$
given by
\[
\lK x,y,z \rK=-[[x,y],z],
\]
and its standard $5$\dash graded embedding $\G=\G(\LL_1,\LL_{-1})$
is canonically isomorphic to the graded Lie subalgebra $[\LL,\LL]+k\zeta$ of $\LL$.
\end{lemma}
\begin{proof}
By~\cite[Lemma 4.3.3]{BdMS} $(\LL_1,\LL_{-1})$ is a Kantor pair, and $\G(\LL_1,\LL_{-1})$ is as required.
It remains to prove centrality and simplicity.
By~\cite[Lemma 4.1.6]{BdMS} the $k$-Lie algebra $[\LL,\LL]$ is central simple, and differs from $\LL$ only in the
grading $0$ component. Let $(\LL_1,\LL_{-1})'$ be the Kantor pair which differs from $(\LL_1,\LL_{-1})$
by the sign of the triple product, i.e. $\lK x,y,z \rK'=[[x,y],z]$. Then
$[\LL,\LL]$ envelopes the Kantor
pair $(\LL_1,\LL_{-1})'$ in the sense of~\cite{AllFauSmi}. Then by~\cite[Theorem 4.20]{AllFauSmi} $(\LL_1,\LL_{-1})'$
is central simple. Clearly, this is equivalent to $(\LL_1,\LL_{-1})$ being central simple.
\end{proof}

Now we can also establish the converse, namely, that any simple Kantor pair
over a field
of characteristic $\neq 2,3$
arises from an isotropic
simple algebraic group.

\begin{theorem}\label{thm:kp}
Let $(K_+,K_-)$ be a central simple Kantor pair over a field $k$ of characteristic $\neq 2,3$, and
let $\G=\G(K_+,K_-)$ be its standard $5$-graded embedding.
Then
the algebraic $k$\dash group $\GG=\Aut(\G)^\circ$
is an adjoint absolutely simple group of $k$\dash rank $\geq 1$, and $\G\cong [\Lie(\GG),\Lie(\GG)]+k\zeta$.
\end{theorem}
\begin{proof}
By definition, we have $\G_1=K_+$, $\G_{-1}=K_-$.
Let $(\G_1,\G_{-1})'$ be the Kantor pair which differs from $(\G_1,\G_{-1})$
by the sign of the triple product, i.e. $\lK x,y,z \rK'=-\lK x,y,z \rK=[[x,y],z]$.
Since $(\G_1,\G_{-1})$ is central simple, $(\G_1,\G_{-1})'$ is also central simple.
Let $\mathcal{K}=\mathfrak{K}((\G_1,\G_{-1})')$ be the 5-graded Lie algebra associated to $(\G_1,\G_{-1})'$ in~\cite[\S~4.3]{AllFauSmi}.
By construction, $\G=k\zeta+\mathcal{K}$.
By~\cite[Corollary 4.14]{AllFauSmi} the Lie algebra $\mathcal{K}$ is central simple.
Then by Theorem~\ref{thm:Che-type} $\mathcal{K}$ is of Chevalley type.
Then by Theorem~\ref{thm:Che-alg} it is algebraic. Then by~\cite[Theorem 4.1.8]{BdMS}
$\GG=\Aut(\mathcal{K})^\circ$ is an adjoint absolutely simple group of $k$\dash rank $\geq 1$, and
$\mathcal{K}\cong [\Lie(\GG),\Lie(\GG)]$. Then by~\cite[Lemma 4.1.6]{BdMS} we have $\Aut(\G)^\circ\cong\GG$.
\end{proof}

\begin{corollary}\label{cor:ka-degen}
Let $(K_+,K_-)$ be a simple Kantor pair over a commutative ring $k$ such that $2,3\in k^\times$. Then
$(K_+,K_-)$ is non-degenerate in the sense of~\cite{GGLN}.
\end{corollary}
\begin{proof}
Since $(K_+,K_-)$ is simple, $k$ does not have non-trivial ideals, and hence is a field. Extending $k$, we can assume that
$(K_+,K_-)$ is central simple. Then by Theorem~\ref{thm:kp} its standard 5-graded embedding
$\LL$ is a central simple Lie
algebra of Chevalley type, and, in particular, algebraic. Then $\LL$ is non-degenerate by Lemma~\ref{lem:alg-nondegen}.
Then by~\cite[Corollary 2.5]{GGLN} $(K_+,K_-)$ is non-degenerate.
\end{proof}

The following lemmas provides a criterion that a Kantor pair (or, equivalently, its 5-graded standard embedding)
originates from a structurable algebra via Allison's construction. It is a combination of~\cite[Corollaries 14 and 15]{AF99}.
This result is crucial for the classification of 5-gradings on simple Lie algebras of Chevalley type that correspond
to structurable algebras (Theorems~\ref{thm:A-e} and~\ref{thm:5-class} below).

\begin{lemma}\label{lem:struct-kan}
Let $R$ be a commutative ring with $2,3\in R^\times$. Let $\G$ be a 5-graded Lie algebra over $R$ which is
the standard 5-graded embedding of a Kantor pair $(\G_{-1},\G_1)$ over $R$.
Then $(\G_{-1},\G_1)$ is a Kantor pair associated to a structurable algebra $\A$ over $R$
if and only if there exist $u\in\G_{1}$, $v\in\G_{-1}$ such that
$\zeta=[u,v]$ is the grading derivation of $\G$.
\end{lemma}
\begin{proof}
Assume first that the Kantor pair $(\G_{-1},\G_1)$ is isomorphic to $(\A,\A)$, and $\G$ is graded-isomorphic to $K(\A)+k\zeta$.
Let $1\in\G_1$, $\hat 1\in\G_{-1}$ denote
the images of $1\in\A$. By the definition of $K(\A)$, the element $V_{1,1}=[1_+,1_-]\in K(\A)_0$ acts as the grading derivation on $K(\A)$. Then
its image $[1,\hat{1}]$ in $\G$ coincides with $\zeta$.

Conversely, assume that $\zeta=[u,v]$ for some $u\in\G_{1}$, $v\in\G_{-1}$.
By~\cite[Corollary 14]{AF99} an element $(x,0)\in\G_1\oplus\G_2$ is 1-invertible if and only if there is $\hat  x\in\G_{-1}$
such that $V_{x,\hat  x}=2\id_{\G_1}$, $V_{\hat x,x}=2\id_{\G_{-1}}$. Since $[u,v]=\zeta$, we have $V_{u,v}=-\id_{\G_1}$
and $V_{v,u}=-\id_{\G_{-1}}$. Then $x=u$ is 1-invertible with $\hat x=-2v$. Since $(\G_1,\G_{-1})$ contains a 1-invertible
element of the form $(x,0)$, by~\cite[Corollary 15]{AF99} it is isomorphic to a Kantor pair associated with
a structurable algebra.
\end{proof}

\begin{lemma}\label{lem:A-zeta}
Let $k$ be a field of characteristic different from $2,3$. Let $\GG$ be an adjoint simple algebraic group over
$k$. Let $\LL=\Lie(\GG)$ be its Lie algebra, and let $\LL=\bigoplus\limits_{i=-2}^2\LL_i$ be any $5$\dash grading on
$\LL$ such that $\LL_1\oplus\LL_{-1}\neq 0$, and let $\zeta\in\LL_0$ be the grading derivation.
Then $\zeta=[u,v]$ for some $u\in\LL_{1}$, $v\in\LL_{-1}$ if and only if there is a structurable algebra $\A$ over $k$
such that $[\LL,\LL]$ is graded-isomorphic to $K(\A)$.
\end{lemma}
\begin{proof}
Assume first that $\LL$ is graded-isomorphic to $K(\A)$. Let $1\in\LL_1$, $\hat 1\in\LL_{-1}$ denote
the images of $1\in\A$. Then $[1,\hat{1}]\in\LL_0$ acts as the grading derivation $\zeta$ on $[\LL,\LL]$. Since $\LL\cong\Der([\LL,\LL])$
by~\cite[Lemma 4.1.6]{BdMS}, $[1,\hat{1}]=\zeta$.

Conversely, assume that $\zeta=[u,v]$ for some $u\in\LL_{1}$, $v\in\LL_{-1}$.
Consider the Kantor pair $(\LL_{-1},\LL_1)$
with the triple product operation $\LL_\sigma\times \LL_{-\sigma}\times \LL_{\sigma}\to \LL_\sigma$
given by
\[
\lK x,y,z \rK=-[[x,y],z].
\]
By Lemma~\ref{lem:4.3.3} its standard 5-graded embedding $\G=\G(\LL_1,\LL_{-1})$
is canonically isomorphic to the graded subalgebra $[\LL,\LL]+k\zeta$ of $\LL$.
Then by Lemma~\ref{lem:struct-kan} the pair $(\LL_{-1},\LL_1)$ is associated to a structurable algebra $\A$. Then $K(\A)$
is graded-isomorphic to $[\G,\G]\cong [\LL,\LL]$.
\end{proof}

\section{Classification of $5$-gradings that correspond to structurable algebras}\label{sec:5-class}

\begin{definition}\label{def:rootsys}
Let $\GG$ be an algebraic group over a field $k$ and $\GT\subseteq\GG$ be a split $n$-dimensional
$k$\dash subtorus of $\GG$.
Let $X^*(\GT)\cong \ZZ^n$ be the group of characters of $\GT$, and let
\[
\Lie(\GG)=\bigoplus_{\alpha\in X^*(\GT)}\Lie(\GG)_\alpha
\]
be the $\ZZ^n$-grading on $\Lie(\GG)$ induced by the adjoint action of $\GT$. We call
\[
\Phi(\GT,\GG)=\{\alpha\in X^*(\GT) \mid \Lie(\GG)_\alpha\neq 0\}
\]
the set of \emph{roots of $\GG$ with respect to $\GT$}.
\end{definition}

If $\GG$ is a reductive algebraic group over $k$ and $\GT$ is a maximal split $k$\dash subtorus of $\GG$, then
$\Phi(\GT,\GG)\setminus \{0\}$ is a root system in the sense of Bourbaki~\cite{BorelTits,Bu}.
By abuse of language, we call $\Phi(\GT,\GG)$ a root system of $\GG$.

Let $\Phi$ be a root system and $\Pi\subseteq\Phi$ be a system of simple roots. For any $\alpha\in\Phi$
we write
\[
\alpha=\sum\limits_{\beta\in\Pi}m_\beta(\alpha)\beta,
\]
where the coefficients $m_\beta(\alpha)$ are all non-negative, or all non-positive. Once $\Pi$
is fixed, we denote the corresponding sets of positive and negative roots by $\Phi^\pm$.

Recall that two parabolic
$k$-subgroups $\GP_{\pm }$ of $\GG$ are called opposite, if $\GP_+\cap\GP_-$ is their common Levi subgroup.
If this is the case, then there is a maximal split $k$\dash subtorus $\GT$ of $\GG$ such that
\[
\GT\subseteq\GP_+\cap\GP_-,
\]
and a system of simple roots
$\Pi\subseteq\Phi$ and a non-empty subset $J\subseteq\Pi$ such that, if one defines
$\Psi=\Phi^+\cup\bigl(\Phi\cap\ZZ(\Pi\setminus J)\bigr)$, then
\begin{equation}\label{eq:lieP}
\begin{aligned}
&\Lie(\GP_+)=\bigoplus\limits_{\alpha\in\Psi}\Lie(G)_\alpha,\qquad
&\Lie(\GP_-)=\bigoplus\limits_{\alpha\in-\Psi}\Lie(G)_\alpha.\\
\end{aligned}
\end{equation}
Conversely, if $\GT$ is maximal split $k$\dash subtorus of $\GG$, $\Phi=\Phi(T,G)$, and
$\Psi\subseteq\Phi$ is a subset such that $0\in\Psi$, $\Psi$ is closed under the (partially defined)
addition of elements
of $\Phi$, and $\Psi\cup(-\Psi)=\Phi$, then there is a unique pair of opposite parabolic $k$-subgroups $\GP_{\pm }$ of $\GG$
satisfying~\eqref{eq:lieP}~\cite[Exp.\@~XXVI, Prop. 6.1]{SGA3}. In particular, every such
subset $\Psi$ has the form $\Psi=\Phi^+\cup\bigl(\Phi\cap\ZZ(\Pi\setminus J)\bigr)$ as above.

\begin{definition}
If $k$ is an algebraically closed field, then
the set $t(\GP_+)=\Pi\setminus J$ is called the \emph{type} of $\GP_+$; it is a system of simple roots of the
root system of the Levi subgroup $\GP_+\cap\GP_-$ of $\GP_+$. If $k$ is not algebraically closed, and $\bar k$ is
an algebraic closure of $k$, then we define $t(\GP_+)$ to be the type
of $(\GP_+)_{\bar k}$.
\end{definition}

A subset $\Psi\subseteq\Phi$ such that $0\in\Psi$, $\Psi$ is closed under addition, and $\Psi\cup(-\Psi)=\Phi$
arises, in particular, if we are given a non-trivial $\ZZ$-grading on $\Lie(G)$. Namely, one takes
$$
\Psi=\{\alpha\in\Phi(T,G)\ |\ \Lie(G)_\alpha\subseteq\bigoplus_{i\ge 0}\Lie(G)_i\}.
$$
Thus, to any such grading one may associate a pair of opposite parabolic subgroups.

In order to classify $5$-gradings that correspond to structurable algebras, we rely on the theory of nilpotent
orbits in Lie algebras of simple algebraic groups over an algebraically closed field. This theory
originates from the work of E. Dynkin~\cite{Dyn},
with further developements by B. Kostant,
G. E. Wall, R. W. Richardson, T.A. Springer, R. Steinberg, G. B. Elkington,
P. Bala and R. Carter, K. Pommerening,
and many others. As a result, the specific statements that we need are seriously scattered in the literature, and
we cite them according to more recent sources where they are stated in a more explicit form.

We recall the essence of Dynkin's classification of nilpotent elements in complex simple Lie algebras.
Let $\LL_{\CC}$ be a simple Lie algebra over $\CC$, and let
$e\in\LL_{\CC}$ be a nilpotent element. By the Jacobson--Morozov theorem,
$\LL_{\CC}$ contains an $sl_2$-triple of the form $\{e,h,f\}$. Let $\mathcal{H}\le\LL_{\CC}$ be a Cartan subalgebra of $\LL_{\CC}$ containing
$h$. Let $\Phi$ be the root system of $\LL_{\CC}$, and let $e_\alpha$, and $h_\alpha=[e_\alpha,e_{-\alpha}]$, $\alpha\in\Phi$,
be the standard root vectors in a Chevalley
basis of $\LL_{\CC}$ with respect to $\mathcal{H}$. There is a choice of a system of simple roots $\Pi\subseteq\Phi$ such that
for any $\alpha\in\Pi$ one has $\alpha(h)\ge 0$. Dynkin established that, morever, $\alpha(h)\in\{0,1,2\}$.
The Dynkin diagram of $\Phi$ with the integers $\alpha(h)$ associated to the nodes
corresponding to roots $\alpha\in\Pi$ is called a \emph{weighted Dynkin diagram} of $e$. It is uniquely
determined by $e$, and the weighted diagrams of two nilpotents $e,e'$ coincide if and only if
$e$ and $e'$ are conjugate by an inner automorphism of $\LL_{\CC}$~\cite[Theorems 8.1 and 8.3]{Dyn}.

Let $\LL_{\ZZ}$ be the $\ZZ$-Lie subalgebra of $\LL_{\CC}$ generated by all $e_\alpha$ and $h_\alpha$, $\alpha\in\Phi$.
Then $\LL_{\ZZ}$ is a $\ZZ$-form of $\LL_{\CC}$, i.e. $\LL_{\CC}=\LL_{\ZZ}\otimes_{\ZZ}\CC$.
Let $k$ be any algebraically closed field, and let $G^{sc}$ be a simply connected simple algebraic
group over $k$ of the same type $\Phi$ as $\LL_{\CC}$, and let $G^{ad}$ be the corresponding
adjoint group. Then $\Lie(G^{sc})\cong\LL_{\ZZ}\otimes_{\ZZ} k$, see
e.g.~\cite[Theorem 23.72]{Milne}.
Define the cocharacter $\lambda:\Gm\to G^{ad}\le \Aut(\Lie(G^{sc}))$ in such a way that
$\lambda(t)\cdot e_{\alpha}=t^{\alpha(h)}e_\alpha$
and $\lambda(t)\cdot h_\alpha=h_\alpha$ for any $\alpha\in\Pi\cup(-\Pi)$ and $t\in k^\times$.
Thus, we associate to any weighted Dynkin diagram a $k$-cocharacter $\lambda:\Gm\to G^{ad}$.

\begin{notation}
Let $G$ be a reductive algebraic group over a field $k$, and let $\lambda:\Gm\to G$ be a cocharacter of $G$ over $k$. Then
$\lambda$ induces a $\ZZ$-grading on $\Lie(G)$, and we denote by $\Lie(G)(\lambda,i)$ the $i$-th component of this grading,
i.e.
$$
\Lie(G)(\lambda,i)=\{ v\in\Lie(G)\ |\ \lambda(t)\cdot v=t^iv\mbox{ for any }t\in\Gm(k)=k^\times\}.
$$
We denote by $P(\lambda)$ and $P(-\lambda)$ the unique pair of (not necessarily proper) opposite
parabolic subgroups  in $G$
such that
$$
\Lie(P(\lambda))=\bigoplus_{i\ge 0}\Lie(G)(\lambda,i)\quad\mbox{and}\quad
\Lie(P(-\lambda))=\bigoplus_{i\ge 0}\Lie(G)(\lambda,-i).
$$
We denote by $C_G(\lambda)$ the centralizer
of $\lambda(\Gm)$ in $G$; this is a Levi subgroup $P(\lambda)\cap P(-\lambda)$ of $P(\lambda)$.
\end{notation}

The classification of weighted Dynkin diagrams corresponding to nilpotents involves the notion of a distinguished parabolic
subgroup. Note that the classification of types of distinguished parabolic subgroups, as defined below, is independent
of the characteristic of the base field.

\begin{definition}\cite[\S~2.6]{LieSe-conj}
Let $G$ be a semisimple algebraic group over an algebraically closed field $k$, and let $\Phi=\Phi(T,G)$ be the root system of $G$.
A parabolic subgroup $P$ of $G$ is called \emph{distinguished} if
\begin{equation}\label{eq:dist}
\dim L_P=\dim \hspace{-2ex}\bigoplus\limits_{\substack{\alpha\in\Phi \colon \\[.6ex] \sum\limits_{\beta\in J}m_\beta(\alpha)=1}}
\hspace*{-2ex}\Lie(G)_\alpha,
\end{equation}
where $L_P$ is a Levi subgroup of $P$, and $\Pi\setminus J$
is the type of $P$ in a system of simple roots $\Pi$ of $\Phi$.
\end{definition}

\begin{theorem}~\cite{Premet03,LieSe-conj}\label{thm:nilp-orbits}
Let $G$ be an adjoint simple algebraic group over an algebraically closed field $k$ of type $\Phi$ such that
$\Char(k)\neq 2$ if $\Phi=B_l,C_l$ ($l\ge 2$) or $\Phi=D_l$ ($l\ge 4$),
and $\Char(k)\neq 2,3$ if $\Phi=E_6,E_7,E_8,G_2,F_4$.
\begin{enumerate}
\item Let $\lambda:\Gm\to G$ be a $k$-cocharacter of $G$ corresponding to a weighted Dynkin diagram of a nilpotent element
in a complex simple Lie algebra of the same type as $G$.
Then
\begin{enumerate}[(i)]
\item  $C_G(\lambda)$ has a unique dense open orbit $V$ in $\Lie(G)(\lambda,2)$.

\item  For any $e\in V(k)$, $C_G(e)\le P(\lambda)$.

\item  For any $e\in V(k)$, let $C(\lambda,e)=C_G(\lambda)\cap C_G(e)$. Then $C(\lambda,e)^\circ$ is a reductive subgroup of $G$, and
for any maximal torus $S$ of $C(\lambda,e)$ the parabolic subgroup $Q=P(\lambda)\cap H$ is a distinguished
parabolic subgroup of the semisimple group $H=[C_G(S),C_G(S)]$. Moreover, $\lambda(\Gm)\le H$,
and $e$ lies in the dense open orbit of $C_H(\lambda)$ in $\Lie(H)(\lambda,2)$.
\end{enumerate}
\item  Conversely, for any nilpotent element $e\in\Lie(G)$ there is a cocharacter $\lambda:\Gm\to G$ as above, such that
$e$ belongs to the unique dense open orbit of $C_G(\lambda)$ in $\Lie(G)(\lambda,2)$.
\end{enumerate}
\end{theorem}
\begin{proof}
(1) Set $\LL=\Lie(G)$, $P=P(\lambda)$, $L_P=C_G(\lambda)$, and $L_Q=C_H(\lambda)$ for short. Clearly, $L_P$ is a Levi subgroup of $P$ and
$L_Q$ is a Levi subgroup of $Q$.

Assume first that $G$ is not of type $E_8$ if $\Char(k)=5$.
Then the characteristic of $k$ is good for $G$.
There is reductive group $\tilde G$ over $k$ such that $[\tilde G,\tilde G]$ is the simply connected group isogenous to $G$,
and $\Lie(\tilde G)$ admits a non-degenerate $\tilde G$-invariant trace form, see~\cite[2.3]{Premet03}.
To prove our theorem, clearly, we can replace $G$ by $\tilde G$; this makes other results of Premet applicable.
By the discussion before~\cite[Theorem 2.3]{Premet03} $L_P$ has a dense open orbit $V$ in
$\LL(\lambda,2)$, and  for any $e\in V(k)$, $C_G(e)\le P$.
By~\cite[Theorem 2.3 (iii)]{Premet03} $C(\lambda,e)$ is a reductive subgroup of $G$.
By~\cite[Theorem 2.3 (ii)]{Premet03} $C_G(e)\le P$.
The remaining statements of (iii) follow from~\cite[Proposition 2.5]{Premet03}.

Assume that $G$ is of type $E_8$ and $\Char(k)\neq 2,3$. Consider the table~\cite[Table 22.1.1]{LieSe-conj}.
By~\cite[Theorem 15.1]{LieSe-conj} the weighted diagrams of complex nilpotent elements are exactly the ones in
the 2nd column of this table, and, conversely, for any $\lambda$ corresponding to such a diagram, and any field $k$ as above,
there is a nilpotent element $e\in\LL(\lambda,2)$ such that $e^P$ is dense in $\bigoplus_{i\ge 2}\LL(\lambda,i)$
and $C_G(e)\le P$. This implies that $e^{L_P}=V$ is a dense open orbit of $L_P$ in $\LL(\lambda,2)$. Since any two such orbits
would intersect, this orbit is unique.
Clearly, it is enough to establish (iii) for this particular element $e\in V(k)$.

By~\cite[Theorem 1 (c)]{LieSe-conj} $C_G(e)=C(\lambda,e)R_u(C_G(e))$, where
$R_u(C_G(e))$ is the unipotent radical of $C_G(e)$ and
 $C(\lambda,e)^\circ$ is a reductive group (note that, contrary to our conventions, in~\cite{LieSe-conj}
reductive groups are not required to be connected).
Let $S$ be a maximal torus of $C(\lambda,e)$, then
$C_G(S)$ is a Levi subgroup of a parabolic subgroup of $G$ by~\cite[Lemma 2.2]{LieSe-conj}. Set $H=[C_G(S),C_G(S)]$.
Clearly, $e\in \LL(\lambda,2)\cap\Lie(H)$, since $e$ is nilpotent and belongs to $\Lie(C_G(S))$.
By the proof of~\cite[Lemma 2.13]{LieSe-conj}
$e$ is a distinguished element of $H=[C_G(S),C_G(S)]$,
i.e. $C_{H}(e)^\circ$ is a unipotent group. Since $S\le C_G(e)$, this implies that $S=\Cent(C_G(S))^\circ$.
On the other hand, by the actual statement of~\cite[Lemma 2.13]{LieSe-conj},
$C_G(S)$ is conjugate to the Levi subgroup $\bar L$ of a parabolic subgroup of $G$ used in the
original construction of the 1-dimensional torus $\lambda(\Gm)=T$ given in~\cite[Lemma 15.3 (i)]{LieSe-conj}.
Let $\bar S=\Cent(\bar L)^\circ$, then $\bar S\le C_G(e)$.
Since $S$ and $\bar S$ are conjugate in $G$, they have the same dimension, and hence they are both maximal
tori in $C_G(e)$. Moreover, they are both contained in $C(\lambda,e)^\circ\le C_G(e)^\circ$,
hence they are conjugate in this group,
i.e. by an element centralizing $T$. Then the remaining statements of our claim (iii) follow from the corresponding
properties of $T$ with respect to $\bar L$ stated in~\cite[Lemma 15.3 (i)]{LieSe-conj}.

(2) Let $e\in\Lie(G)$ be any nilpotent.
The classification of nilpotent classes~\cite[Theorem 1 c)]{LieSe-conj} implies that
there is a cocharacter $\lambda:\Gm\to G$, that corresponds
to a weighted Dynkin diagram, and such that $e$ belongs to the dense open orbit of $C_G(\lambda)$ in $\Lie(G)(\lambda,2)$.
Moreover, the explicit classification
of occurring weighted Dynkin diagrams~\cite[Theorem 3.1; Tables 22.1.1-22.1.5]{LieSe-conj} is independent
of the ground field under the assumption $\Char k\neq 2,3$, hence any such weighted
Dynkin diagram is a diagram of a complex nilpotent.
\end{proof}

\begin{lemma}\label{lem:H-2}
In the setting of Theorem~\ref{thm:nilp-orbits} (1) (iii), assume moreover that $H$ is not of type $E_8$ if $\Char(k)=5$.
Then
$[\Lie(H)(\lambda,-2),\, e]=\Lie(H)(\lambda,0)$.
\end{lemma}
\begin{proof}
Let $\Psi$ be the root system of $H$, let $\Sigma$ be a system of simple roots of $\Psi$,
and let $J\subset\Sigma$ be the set of simple
roots corresponding to the parabolic subgroup $Q=P(\lambda)\cap H$ of $H$.
Since $Q$ is distinguished, one has $\lambda(\alpha)=2$ for all $\alpha\in J$,
see~\cite[Lemma 10.3]{LieSe-conj}. Hence one has
$$
\dim\bigl(\Lie(H)(\lambda,2)\bigr)=\dim\bigl(\Lie(H)(\lambda,-2)\bigr)=\dim\bigl(\Lie(H)(\lambda,0)\bigr).
$$
Hence it remains to prove that
$[u,e]\neq 0$ for any $0\neq u\in\Lie(H)(\lambda,-2)$. If $(\Psi,\Char(k))\neq (E_8,5)$, then this follows from~\cite[Theorem 2.3 (iv)]{Premet03}.
\end{proof}

From now and until the end of this section, assume that $k$ is a field of characteristic $\neq 2,3$, and $G$ is an adjoint simple algebraic group over $k$.
Denote by $\bar k$ an algebraic closure of $k$.
Let $\LL=\Lie(G)$ be the Lie algebra of $G$, and let $\LL=\bigoplus\limits_{i\in\ZZ}\LL_i$ be any $\ZZ$-grading on $\LL$
such that $\LL\neq\LL_0$. By~\cite[Lemma 4.1.6]{BdMS} one has
$G=\Aut(\Lie(G))^\circ$, hence there is a unique
closed embedding of a 1-dimensional split
$k$\dash torus $\lambda:\Gm\to G$,
such that $\LL_i=\Lie(G)(\lambda,i)$ for all $i\in\ZZ$.

\begin{theorem}\label{thm:A-e}
Let $\LL=\bigoplus\limits_{i=-2}^2\LL_i$ be a $5$-grading on the Lie algebra $\LL=\Lie(G)$ of $G$ such that
$\LL_1\oplus\LL_{-1}\neq 0$, and let $\lambda:\Gm\to G$
be the corresponding cocharacter of $G$. Let $\Delta$ be the weighted Dynkin diagram of the same root system type as $G_{\bar k}$,
such that simple roots in $t(P(\lambda))$
have weight 0, and roots in $\Pi\setminus t(P(\lambda))$ have weight 2.
Then $[\LL,\LL]$ is graded-isomorphic to $K(\A)$ for a structurable algebra $\A$ over
$k$ if and only if $\Delta$ is a weighted Dynkin diagram
of a nilpotent element
in a complex simple Lie algebra of the same type.
\end{theorem}
\begin{proof}
Consider the cocharacter $2\lambda:\Gm\to G$, so that $\LL_i=\Lie(G)(2\lambda,2i)$, $i\in\ZZ$. Then $\Delta$
is the weighted Dynkin diagram corresponding to $2\lambda$ over $\bar k$.

Assume the condition on $\Delta$ is satisfied. We show that there are $u\in\LL_{-1}$, $e\in\LL_1$ such that
$[u,e]=\zeta$, the grading derivation of $\LL$; then $[\LL,\LL]$ is graded-isomorphic to $K(\A)$ by Lemma~\ref{lem:A-zeta}.
Let $1+\epsilon\in\Gm(k[\epsilon])$ be the unit element of $\Lie(\Gm)(k)\cong k$. Then
$2\zeta=2\lambda(1+\epsilon)\in\Lie(G)$
is the grading derivation of $\Lie(G)$ with respect to $2\lambda$.
We have
$$
2\zeta\in\Lie(G)(2\lambda,0)\cap\Lie(2\lambda(\Gm))\subseteq\Lie(H)(2\lambda,0),
$$
where $H$ is as in Theorem~\ref{thm:nilp-orbits}.
By Lemma~\ref{lem:H-2} there is a dense open orbit $V\subseteq\LL_1=\Lie(G)(2\lambda,2)$ of $L_P$, such that for any $e\in V(k)$ one has
$[\Lie(H)(2\lambda,-2),\, e]=\Lie(H)(2\lambda,0)$.
In particular, $\zeta\in [\Lie(G)(2\lambda,-2),e]=[\LL_{-1},e]$.

Now let $k$ be not necessarily algebraically closed, and let $\bar k$ be its algebraic closure. Note that $P(\pm\lambda)$ and $C_G(\lambda)=P(\lambda)\cap P(-\lambda)$
are defined over $k$. Since $C_G(\lambda)\times_k \bar k$ has a unique dense open orbit $V$ in $\LL_1\otimes_k \bar k$,
the open subvariety $V$ of the affine space $\LL_1$ is defined over $k$ (although the action of $C_G(k)$ on it does not have to be transitive).
If $k$ is infinite, then $V(k)\neq\emptyset$ just because $V$ is
an open subvariety of an affine space. If $k$ is finite, then $V(k)\neq\emptyset$ by~\cite[2.7]{SpSt-conj}, since
$V$ is a homogeneous space for $C_G(\lambda)$.
Thus, there is a $k$-point $e\in V(k)$. Then $\zeta\in [\LL_{-1},e]$,
since the same holds over $\bar k$.

Next, assume that $[\LL,\LL]$ is graded-isomorphic to $K(\A)$, and show that $\Delta$ is a weighted Dynkin diagram
of a nilpotent element
in a complex simple Lie algebra of type $\Phi$. We can assume that $k$
is algebraically closed without loss of generality.
Let $e=1_+\in\LL_1$ and $f=1_-\in\LL_{-1}$ be the elements representing the unit of the structurable algebra $\A$. By the very definition of $K(\A)$,
we conclude that $[e,f]=\zeta\in\LL_0$. Furthermore, for any $0\neq x\in\A$ one has $V_{1,x}(1)=2\bar x-x\neq 0$,
since $2\bar x=x$ implies $2x=\bar x$ and $3\bar x=0$, whence $x=0$.

By Theorem~\ref{thm:nilp-orbits} there is a cocharacter $\lambda':\Gm\to G$ that corresponds
to a weighted Dynkin diagram of complex nilpotent,
and such that $e$ belongs to the dense open orbit of $C_G(\lambda')$ in $\Lie(G)(\lambda',2)$.
Let $\zeta'$ be the grading derivation of $\LL$ corresponding to $\lambda'$. Then $[\zeta',e]=2e$.
Assume for the moment that the subgroup $H$ corresponding to $\lambda'$ and $e$ is not of type $E_8$ if $\Char k=5$.
Then by Lemma~\ref{lem:H-2} there is $f'\in\Lie(G)(\lambda',-2)$ such that $[e,f']=\zeta'$. We have
$\lambda(\Gm),\lambda'(\Gm)\le N_G(k\cdot e)$. Hence after conjugating $\lambda'(\Gm)$ by an element of
$C_G(e)(k)\le N_G(k\cdot e)(k)$
we can assume that $\lambda'(\Gm)$ and $\lambda(\Gm)$ lie in the same maximal torus of $N_G(k\cdot e)$, and thus
centralize each other. In particular, $\zeta'\in\LL_0$, and hence without loss of generality $f'\in\LL_{-1}$.
Then $2\zeta-\zeta'=[e,2f-f']$ and $[2\zeta-\zeta',e]=0$. In other words, $[[e,2f-f'],e]=0$. However, $[e,2f-f']$ acts
on $\LL_1$ as $V_{1,2f-f'}$ acts on $\A$, whence $2f-f'=0$ by the above computation. Thus $2\zeta=\zeta'$, and we are done.

Assume that $H=G$ has type $E_8$; then $P(\lambda')$ is a distinguished parabolic subgroup
of $G$. We claim that this case cannot occur in our setting. The
types of distinguished parabolic subgroups are listed in~\cite[Table 13.2]{LieSe-conj}. Denote by
$\LL=\bigoplus_{i\in\ZZ}\LL_{[i]}$ the grading on $\LL$ induced by $\lambda'$.
In all cases,
one readily sees that $\LL_{[2i]}\neq 0$ for all $1\le i\le k$, where $k\ge 5$. Then
by~\cite[proof of Propositions 13.4 and 13.5]{LieSe-conj}, under the assumption $\Char k\neq 2,3$ there is
a nilpotent $e'\in\LL_{[2]}$ such that ${e'}$ lies in the dense orbit of $C_G(\lambda')$ in $\LL_{[2]}$,
and one has $\bigl(\ad(e')|_{\bigoplus_{i\ge 0}\LL_{[i]}}\bigr)^5\neq 0$. Clearly, $e$ and $e'$ are $C_G(\lambda')$-conjugate.
However, since $e\in\LL_1$, and $\LL$ is $5$-graded, one has $\ad(e)^5=0$, hence $\ad(e')^5=0$ as well,
a contradiction.
\end{proof}

\begin{lemma}\label{lem:5-lambda}
Let $\LL=\bigoplus_{i=-2}^2\LL_i$ be a $5$-grading on the Lie algebra $\LL=\Lie(G)$ of $G$ such that
$\LL_1\oplus\LL_{-1}\neq 0$, and let $\lambda:\Gm\to G$
be the corresponding cocharacter of $G$. Let $\tilde\alpha$ be the root of maximal height in $\Phi$
with respect to $\Pi$. Let $J=\Pi\setminus t(P(\lambda))$, then $J$ satisfies
(a) $J=\{\alpha_1\}$ or (b) $J=\{\alpha_1,\alpha_2\}$.
If (a) holds, then $m_{\alpha_1}(\tilde\alpha)=1$ or $2$. If (b) holds, then
$m_{\alpha_1}(\tilde\alpha)=m_{\alpha_2}(\tilde\alpha)=1$. In both cases
$\lambda(t)\cdot e_\alpha=t e_\alpha$  for any $\alpha\in J$ and $t\in k^\times$.
\end{lemma}
\begin{proof}
By~\cite[Lemma 4.2.2]{BdMS}
the root system $\Phi\cap\ZZ J$
is of type $A_1$, $BC_1$, $A_2$, or $A_1\times A_1$, and
for any $-2\leq i\leq 2$ we have
\begin{equation}\label{eq:Li}
\LL_i = \hspace{-2ex}\bigoplus\limits_{\substack{\alpha\in\Phi \colon \\[.6ex] \sum\limits_{\beta\in J}m_\beta(\alpha)=i}}
\hspace*{-2ex}\Lie(G)_\alpha.
\end{equation}
Then
$\lambda(t)e_\alpha=t e_\alpha$  for $\alpha\in J$ by the definition of $\lambda$. Since
$\sum_{\alpha\in J}m_\alpha(\tilde\alpha)\le 2$, the remaining claims are clear.
\end{proof}

\begin{definition}
Assume that the $\ZZ$-grading $\Lie(G)=\bigoplus_{i\in\ZZ}\Lie(G)_i$ of the Lie algebra of a
reductive algebraic group $G$ corresponds to the cocharacter $\lambda:\Gm\to G$.
We define the \emph{type} of the grading to be the type $t(P(\lambda))$
of the corresponding positive parabolic subgroup $P(\lambda)$.
\end{definition}

\begin{theorem}\label{thm:5-class}
Let $G$ be an adjoint simple algebraic group over a field $k$, $\Char k\neq 2,3$.
Let $\LL=\bigoplus\limits_{i=-2}^2\LL_i$ be a $5$-grading on the Lie algebra $\LL=\Lie(G)$ of $G$ such that
$\LL_1\oplus\LL_{-1}\neq 0$. Then $[\LL,\LL]$ is graded-isomorphic to $K(\A)$ for a structurable $k$-algebra $\A$
if and only if the type of grading is the complement of the set $J$ of simple roots of the root system $\Phi$
of $G$ listed in the following table.
\end{theorem}
{\renewcommand{\arraystretch}{1.3}

\begin{tabular}{r|l}
$\Phi$& $J$\\
\cline{1-2}
$A_l$, $l\ge 1$ & $\{\alpha_i,\alpha_{l+1-i}\}$, $1\le i\le (l+1)/3$;\\
 & $\{\alpha_{l+1/2}\}$, if $l$ is odd\\
\cline{1-2}
$B_l$, $l\ge 2$ & $\{\alpha_i\}$, $1\le i\le (2l+1)/3$\\
\cline{1-2}
$C_l$, $l\ge 3$ & $\{\alpha_{i}\}$, $1\le i\le 2l/3$, $i$ is even; $\{\alpha_l\}$\\
\cline{1-2}
$D_l$, $l\ge 4$ & $\{\alpha_i\}$, $1\le i\le 2l/3$; \\
 & $\{\alpha_{l-1}\}$ and $\{\alpha_l\}$, if $l$ is even\\
\cline{1-2}
$E_6$ & $\{\alpha_1,\alpha_6\}$; $\{\alpha_2\}$\\
\cline{1-2}
$E_7$ & $\{\alpha_1\}$; $\{\alpha_2\}$; $\{\alpha_6\}$; $\{\alpha_7\}$\\
\cline{1-2}
$E_8$ & $\{\alpha_1\}$; $\{\alpha_8\}$\\
\cline{1-2}
$F_4$ & $\{\alpha_1\}$; $\{\alpha_4\}$\\
\cline{1-2}
$G_2$ & $\{\alpha_2\}$\\
\end{tabular}
}

\begin{proof}
Let $\lambda:\Gm\to G$
be the cocharacter of $G$ cooresponding to the grading. By Theorem~\ref{thm:A-e}
it remains to check if $2\lambda$
is a $k$-cocharacter corresponding to a weighted Dynkin diagram of a nilpotent element in the complex case.
By Lemma~\ref{lem:5-lambda} one has $2\lambda(t)\cdot e_\alpha=t^2e_\alpha$ for all $\alpha\in J$ and $2\lambda(t)\cdot e_\alpha=t^0e_\alpha$
for all $\alpha\in\Pi\setminus J$.

If $\Phi$ is of exceptional type, then one readily checks that in all cases $2\lambda$ is as required, since its weighted Dynkin diagram
occurs in the Tables 22.1.1--22.1.5 of nilpotent conjugacy classes in~\cite{LieSe-conj}.

Assume $\Phi$ is of classical type. We use the descriptions of weighted Dynkin diagrams of nilpotent orbits given
in~\cite{CMc}.

{\bf Case $\Phi=A_l$.} In the notation of~\cite[\S~3.6]{CMc}, we have $n=l+1$ and
$h_1-h_n=2|J|\in\{2,4\}$. We need to describe suitable
non-negative partitions $n=\sum_{i=1}^n d_i$. By definition,  $d_i\ge 1$ for $1\le d_i\le k$ and
$d_i=0$ for $k+1\le i\le n$.

If $h_1-h_n=2$, then $d_i\in \{1,2\}$ for all $1\le i\le k$.
Then the sequence $h_1\ge h_2\ge\ldots\ge h_n$
contains numbers $1$ and $-1$ repeated $m\ge 1$ times each, and $k-m$ zeroes.
Since $h_i-h_{i+1}=2$ only for one $i$, we have $k-m=0$ and $n=2k$.
In particular, $l$ is odd, and $J=\{\alpha_{k}\}=\{\alpha_{(l+1)/2}\}$.

If $h_1-h_n=4$, then $d_i\in\{1,2,3\}$ for all $1\le i\le k$.
Then the sequence $h_1\ge h_2\ge\ldots\ge h_n$
contains numbers $2,0,-2$ repeated $m\ge 1$ times each, numbers $1,-1$ repeated $m'$ times each, and
$k-m-m'$ zeroes. Since $h_i-h_{i+1}=2$ for exactly two $i$'s, we conclude that $m'=0$. Then $l+1=n=3m+(k-m)$, hence
$l+1\ge 3m$ and $J=\{\alpha_m,\alpha_{l+1-m}\}$, as required.

{\bf Case $\Phi=B_l$.} In the notation of~\cite[Lemma 5.3.3]{CMc}, we have $n=l$. Nilpotent orbits
are classified by partitions $2n+1=\sum_{i=1}^{2n+1}d_i$ of $2n+1$ in which even parts occur with even
multiplicity~\cite[Theorem 5.1.2]{CMc}. In~\cite[Lemma 5.3.3]{CMc}, $h_1$ is the sum of labels of the
weighted Dynkin diagram, hence $h_1=2$. Then $d_i\in\{1,2,3\}$. If $h_n=2$, then
$h_1=h_2=\ldots=h_n=2$, which is not possible, since there are not enough zeroes.
Hence the label $2$ occurs only for one $i$ between $1$ and $n-1$, and $h_1=\ldots=h_i=2$, $h_{i+1}=\ldots=h_n=0$.
Then partition consists of $i$ numbers $3$ and $2n+1-3i$ numbers $1$. The condition on parity is void.
The set $\{0,\pm h_1,\ldots,\pm h_n\}$ contains $i+(2n+1-3i)$ zeroes, therefore, $2(n-i)\ge i-1$,
which implies $3i\le 2n+1$.

{\bf Case $\Phi=C_l$.} In the notation of~\cite{CMc}, we have $n=l$. Nilpotent orbits
are classified by partitions $2n=\sum_{i=1}^{2n}d_i$ of $2n$ in which odd parts occur with even
multiplicity~\cite[Theorem 5.1.3]{CMc}. In~\cite[Lemma 5.3.1]{CMc}, $h_1+h_n=2$ is the sum of labels of the
weighted Dynkin diagram. Since $2h_n$ is the label of $\alpha_l$, we have $h_n=0$ or $h_n=1$.

If $h_n=1$, then $h_1=1$, and hence $h_i=1$ for all $i$. Then $d_i=2$ for all non-zero $d_i$,
hence the partition is $2n=2+\ldots+2$. The condition on odd parts is satisfied, hence $J=\{\alpha_l\}$ is valid.

If $h_n=0$, then $h_1=2$ and $d_i\in\{1,2,3\}$ for $1\le i\le 2n$. Since $h_i-h_{i+1}=2$ only for one $i$ between $1$
and $n-1$, we conclude that $h_1=\ldots=h_i=2$ and $h_{i+1}=\ldots=h_{n-1}=0$. Then
$2n$ is partitioned into
the sum of $i$ times number $3$, and $2n-3i\ge 0$ times number $1$, where $i$ is even. Summing up, $J=\{\alpha_i\}$
with even $i=2m$, $1\le m\le l/3$, or $J=\{\alpha_l\}$.

{\bf Case $\Phi=D_l$.} In our notation, $l=n$. Nilpotent orbits
are classified by partitions $2n=\sum_{i=1}^{2n}d_i$ of $2n$ in which even parts occur with even
multiplicity, except that "very even"{} partitions with only even parts (each having even multiplicity) correspond to two
different
orbits~\cite[Theorem 5.1.4]{CMc}.
Since all weights of the Dynkin diagram in our seting are even, all numbers $h_1,\ldots,h_n$ have the same parity.
Assume first that the partition is not very even, i.e. the numbers $d_i$ are odd, and the numbers $h_i$ are even.
By~\cite[Lemma 5.3.4]{CMc}
the sum of labels
on the Dynkin diagram equals $h_1+h_{n-1}\in\{2,4\}$.
Since $h_1\ge h_{n-1}$, we have $h_1=2$, $h_{n-1}=0$ or
$h_1=h_{n-1}=2$.
If $h_1=h_{n-1}=2$, then $h_1=h_2=\ldots=h_{n-1}=2$ and $h_n=0$, which is not possible (not enough zeroes).
Hence $h_1=2$, $h_{n-1}=0$. Then $h_n=0$ as well. The partition consists of $i$ times $3$, where $i\ge 1$,
 and $2n-3i$ numbers $1$.
The multiplicity condition is void. Since numbers $3$ produce triples $2,0,-2$, one has $2n\ge 3i$.
This case corresponds to $J=\{\alpha_i\}$, $1\le i\le 2n/3$.

Assume the partition is very even. Then by~\cite[Lemma 5.3.5]{CMc} the sum of labels of vertices $\alpha_n$ and $\alpha_{n-1}$
equals $2$. Then either all other labels are $0$, or there is label $2$ at $\alpha_1$. In the first case we have
$h_1=h_2=\ldots=h_{n-1}$, and since all $d_i$ are even, this implies that the partition is $n$ numbers $2$. Since it
is very even, $n$ is even. In the second case $h_1-h_2=2$, $h_2=\ldots=h_{n-1}$. Since
$h_1\ge h_2\ge h_n$, then $d_1=h_1+1$ can only have multiplicity one, which is wrong. Therefore, this case
does not take place in our setting. Summing up, very even partitions occur with $J=\{\alpha_{n-1}\}$ and
$J=\{\alpha_n\}$ for $n$ even.
\end{proof}

\begin{proof}[Proof of Theorem~\ref{thm:1}]
The implication $\mathit{(3)\implies (2)}$ is obvious.
The implication $\mathit{(2)\implies (1)}$ follows from Theorem~\ref{thm:Che-type}.
It remains to prove $\mathit{(1)\implies (3)}$. By assumption, $\LL$ is a central simple Chevalley type Lie algebra
with a non-trivial $\ZZ$-grading. Let $\GS$ be the grading torus of $\LL$, and let $\varphi:\GS\to\Aut(\LL)$ be the natural
embedding of algebraic $k$-groups. The algebraic group $G=\Aut(\LL)^\circ$ is an adjoint simple algebraic
group, and $\LL=[\Lie(G),\Lie(G)]$ (cf.~\cite[Lemma 4.1.6]{BdMS}).
Since $\GS$ is a 1-dimensional split torus, it follows that $G$ contains a 1-dimensional split torus,
i.e. is isotropic. The possible types of minimal parabolic $k$-subgroups of $G$
are given by the Tits index of $G$. Considering the classification of possible Tits indices~\cite{Tits66,PS-tind},
one readily sees that in all cases there is a parabolic subgroup type $\Pi\setminus J$ listed
in Theorem~\ref{thm:5-class} that is preserved by the $*$-action and contains the type of a minimal parabolic $k$-subgroup
of $G$. Then $G$ contains at least one parabolic $k$-subgroup $P(\lambda)$ whose type is $\Pi\setminus J$~\cite[2.5.4]{Tits66}.
Then by Theorem~\ref{thm:A-e} $\LL$ is graded-isomorphic to $K(\A)$ for a structurable algebra $\A$ over $k$.
\end{proof}

\renewcommand{\refname}{References}

\end{document}